\documentclass[11pt]{article}
\usepackage{amsthm, amsmath, amssymb, amsfonts, url, booktabs, tikz, setspace, fancyhdr, bm}
\usepackage{xcolor}
\usetikzlibrary{positioning, calc, backgrounds, fit}
\usepackage{geometry}
\usepackage{hyperref, enumerate}
\usepackage[shortlabels]{enumitem}
\usepackage[babel]{microtype}
\usepackage[english]{babel}
\usepackage[capitalise]{cleveref}
\usepackage{comment}
\usepackage{bbm}
\usepackage{csquotes}
\usepackage{mathabx}
\usepackage{tikz}
\usetikzlibrary{calc}
\usepackage{graphicx}
\usepackage{float}
\usetikzlibrary{positioning, arrows.meta, shapes.geometric}

\counterwithin{figure}{section}
\definecolor{c1}{HTML}{E41A1C} 
\definecolor{c2}{HTML}{377EB8} 
\definecolor{c3}{HTML}{4DAF4A} 
\definecolor{c4}{HTML}{FF7F00} 
\definecolor{c5}{HTML}{984EA3} 

\newtheorem{theorem}{Theorem}[section]
\newtheorem{prop}[theorem]{Proposition}
\newtheorem{lemma}[theorem]{Lemma}

\newtheorem{conj}[theorem]{Conjecture}

\usetikzlibrary{decorations.pathmorphing}

\theoremstyle{definition}
\newtheorem{defn}[theorem]{Definition}
\newtheorem*{defn-non}{Definition}

\newlist{Case}{enumerate}{2}
\setlist[Case, 1]{%
    label           =   {\bfseries Case \arabic*.},
    labelindent=0em ,labelwidth=1.3cm, labelsep*=1em, leftmargin =!
}
\setlist[Case, 2]{%
    label           =   {\bfseries Subcase \arabic{Casei}.\arabic*.},
    labelindent=-1em ,labelwidth=1.3cm, labelsep*=1em, leftmargin =!
}

\usepackage{todonotes}

\title{Optimal stability results on color-biased Hamilton cycles} 
\author{
Wenchong Chen\thanks{School of Mathematical Sciences, Nankai University,  Tianjin, China Email: 2212161@mail.nankai.edu.cn}
\and
Mingyuan Rong\thanks{School of Mathematical Sciences, University of Science and Technology of China, Hefei,
China.
Email: rong\_ming\_yuan@mail.ustc.edu.cn. Mingyuan Rong was supported by National Key Research and Development Program of China 2023YFA1010201, the NSFC under Grant No. 12125106 and the Excellent PhD Students Overseas Study Program of the University of Science and Technology of China.}
\and
Zixiang Xu\thanks{Extremal Combinatorics and Probability Group (ECOPRO), Institute for Basic Science (IBS), Daejeon, South Korea. Email: zixiangxu@ibs.re.kr. Zixiang Xu was supported by the Institute for Basic Science (IBS-R029-C4).}
}

\begin{document}
\date{}
\maketitle

\begin{abstract}
We investigate Hamilton cycles in edge-colored graphs with \( r \) colors, focusing on the notion of color-bias (discrepancy), the maximum deviation from uniform color frequencies along a cycle. Foundational work by Balogh, Csaba, Jing, and Pluh\'{a}r, and the later generalization by Freschi, Hyde, Lada, and Treglown, as well as an independent work by Gishboliner, Krivelevich, and Michaeli, established that any \(n\)-vertex graph with minimum degree exceeding \( \frac{(r+1)n}{2r} + \frac{m}{2}\) contains a Hamilton cycle with color-bias at least \(m\), and characterized the extremal graphs with minimum degree \(\frac{(r+1)n}{2r}\) in which all Hamilton cycles are perfectly balanced.

We prove the optimal stability results: for any positive integers \(r\ge 2\) and \( m < 2^{-6} r^{2} n,\) if every Hamilton cycle in an \( n \)-vertex graph with minimum degree exceeding \( \frac{n}{2} + 6r^{2}m \) has color-bias less than \( m \), then the graph must closely resemble the extremal constructions of Freschi, Hyde, Lada, and Treglown. The leading term \( \frac{n}{2} \) in the degree condition is optimal, as it is the sharp threshold for guaranteeing Hamiltonicity. Moreover, we show the additive error term \(\Theta(m)\) is also best possible when \(m\) is large and \(r=2\), since weaker condition \(\frac{n}{2}+o(m)\) allow for a counterexample. Notably, the structural stability threshold \( \frac{1}{2} \) lies strictly below the extremal threshold \( \frac{1}{2} + \frac{1}{2r} \) required to force color imbalance. Our proof leverages local configurations to deduce global structure, revealing a rigid combinatorial dichotomy.

\end{abstract}

\section{Introduction}
A \emph{Hamilton cycle} in a graph is a cycle that visits every vertex exactly once. The existence and structure of Hamilton cycles form a central theme in graph theory, and have been extensively studied in various combinatorial settings~\cite{1987Bela,2016MAMS,2009Bill,1952PLMS,2018IMRN,2022RSARandom,2022JCTB,2009Szabo,2011JAMS,2014TRANSAMS,2014PLMS,2013ADV,P}. Meanwhile, the study of \emph{discrepancy}, which quantifies deviation from uniformity or regularity, has emerged as a fundamental topic in various fields including measure theory,
number theory, geometry and theoretical computer science. In discrete setting, discrepancy theory plays a pervasive role in graph theory~\cite{balogh2020discrepanciesgraphs,2021CPCHS,2024BaloghMatching,2022EJCDomago,2024EJCDomagoj,2021SIDMA,2024JCTBDirac,2022RSARandom,2022JCTB}, additive combinatorics~\cite{2022Fox,2024Max,1996JAMS,1964Roth,2016TaoErdosDis}, and beyond; see~\cite{2000ChazelleBook,1999Matousek} for comprehensive treatments. In this paper, we explore the interplay between these two themes by studying how color imbalance manifests within Hamilton cycles in edge-colored graphs.

Throughout this paper, we assume that \( n \) is sufficiently large. For the sake of clarity of presentation, we omit floor and ceiling signs whenever they are not essential. Given \( V_i, V_j \subseteq V(G) \), we write \( E(V_i) := \{ uv \in E(G) : u, v \in V_i \} \) and \( E(V_i, V_j) := \{ uv \in E(G) : u \in V_i,\ v \in V_j \} \), and let $G[V_i,V_j]$ be the bipartite subgraph of $G$ with vertex set $V_i\cup V_j$ and edge set $E(V_i,V_j)$. Given an edge coloring function $\chi:E(G)\rightarrow [r]$, for a color \( k \in [r] \), we let \( \chi^{-1}(k) := \{ e \in E(G) : \chi(e) = k \} \). For any subgraph \( F \subseteq G \) and any integer \( s \), we say that \( F \) is \emph{\( s \)-nearly empty} if it contains no matching of size \( s \). Given a color \( k \in [r] \), we say that \( F \) is \emph{\((s,k)\)-nearly monochromatic} if \( F \setminus \chi^{-1}(k) \) contains no matching of size \( s \).

Given a graph \( G \) with an edge-coloring function \( \chi : E(G) \rightarrow [r] \), and a Hamilton cycle \( H \subseteq G \), the \emph{color-bias} (\emph{discrepancy}) of \( H \) is defined as
\[
d_{\chi}(H) = \max_{i \in [r]} \left|\, |\chi^{-1}(i) \cap E(H)| - \frac{n}{r} \,\right|.
\]
Intuitively, this quantity measures how evenly the \( r \) colors are distributed along the Hamilton cycle. When \( G \) is very dense, it is often hard to avoid Hamilton cycles with large color-bias. This phenomenon was first rigorously analyzed by Balogh, Csaba, Jing and Pluh\'{a}r~\cite{balogh2020discrepanciesgraphs}, who showed that any \(n\)-vertex \(2\)-colored graph with \(\delta(G) = \left(\frac{3}{4} + \varepsilon\right)n\) must contain a Hamilton cycle \( H \subseteq G \) satisfying \(
d_{\chi}(H) \geq \frac{\varepsilon n}{64}.\) This result was subsequently extended by Freschi, Hyde, Lada, and Treglown~\cite{2021SIDMA}, and independently by Gishboliner, Krivelevich, and Michaeli~\cite{2022JCTB}, who established a sharp result for arbitrary \( r \)-colorings. Below, we state the precise result as presented in~\cite{2022JCTB}, which offers a better quantitative bound.
\begin{theorem}[\cite{2021SIDMA,2022JCTB}]\label{thm:2021SIDMA}
Let $r\ge 2$ and $0\le m\le\frac{n}{28r^{2}}$ be integers and let \( G \) be an \( n \)-vertex graph with \( \delta(G) \ge \frac{(r+1)n}{2r} + m \). Then for any edge-coloring \( \chi : E(G) \rightarrow [r] \), there exists a Hamilton cycle \( H \subseteq G \) satisfying \(
d_{\chi}(H) \ge 2m.
\)
\end{theorem}
To demonstrate the sharpness of~\cref{thm:2021SIDMA}, Freschi, Hyde, Lada, and Treglown~\cite{2021SIDMA} provided two extremal constructions showing that the minimum degree threshold \( \delta(G) \geq \frac{r+1}{2r}n \) is best possible.

\paragraph{Extremal construction for general \(r\ge 2\):}  Suppose that \( 2r \mid n \), and partition the vertex set as
\[
V(G) = V_1 \cup V_2 \cup \dots \cup V_r,
\]
where \( |V_i| = \frac{n}{2r} \) for \( 1 \leq i < r \), and \( |V_r| = \frac{(r+1)n}{2r} \). Define the edge set as \(E(G) = \{ uv : u \in V_r,\, v \in V(G) \}.\) That is, every edge has at least one endpoint in \( V_r \). Color the edges so that for \( uv \in E(G) \), we set \( \chi(uv) = i \) if and only if \( u \in V_r \) and \( v \in V_i \) for \( 1 \leq i \leq r \). All remaining edges within \( V_r \) are assigned color \( r \), see the left graph in~\cref{fig:extremal_construction_general_r}. In this construction, each vertex in \( V_i \) (for \( i < r \)) has degree exactly \( |V_r| = \frac{r+1}{2r}n \), and is incident only to color \( i \), while vertices in \( V_r \) are incident to all colors. Since each \( V_i \) for \( i < r \) is an independent set and has size \( \frac{n}{2r} \), any Hamilton cycle \( H \subseteq G \) must use exactly \( 2|V_i| = \frac{n}{r} \) edges of color \( i \). This accounts for \( \frac{(r-1)n}{r} \) edges, so the remaining \( \frac{n}{r} \) edges must be of color \( r \). Therefore, \( H \) contains exactly \( \frac{n}{r} \) edges of each color, and the color-bias satisfies \( d_{\chi}(H) = 0 \).


\begin{figure}
    \centering
    \includegraphics[width=1.02\linewidth]{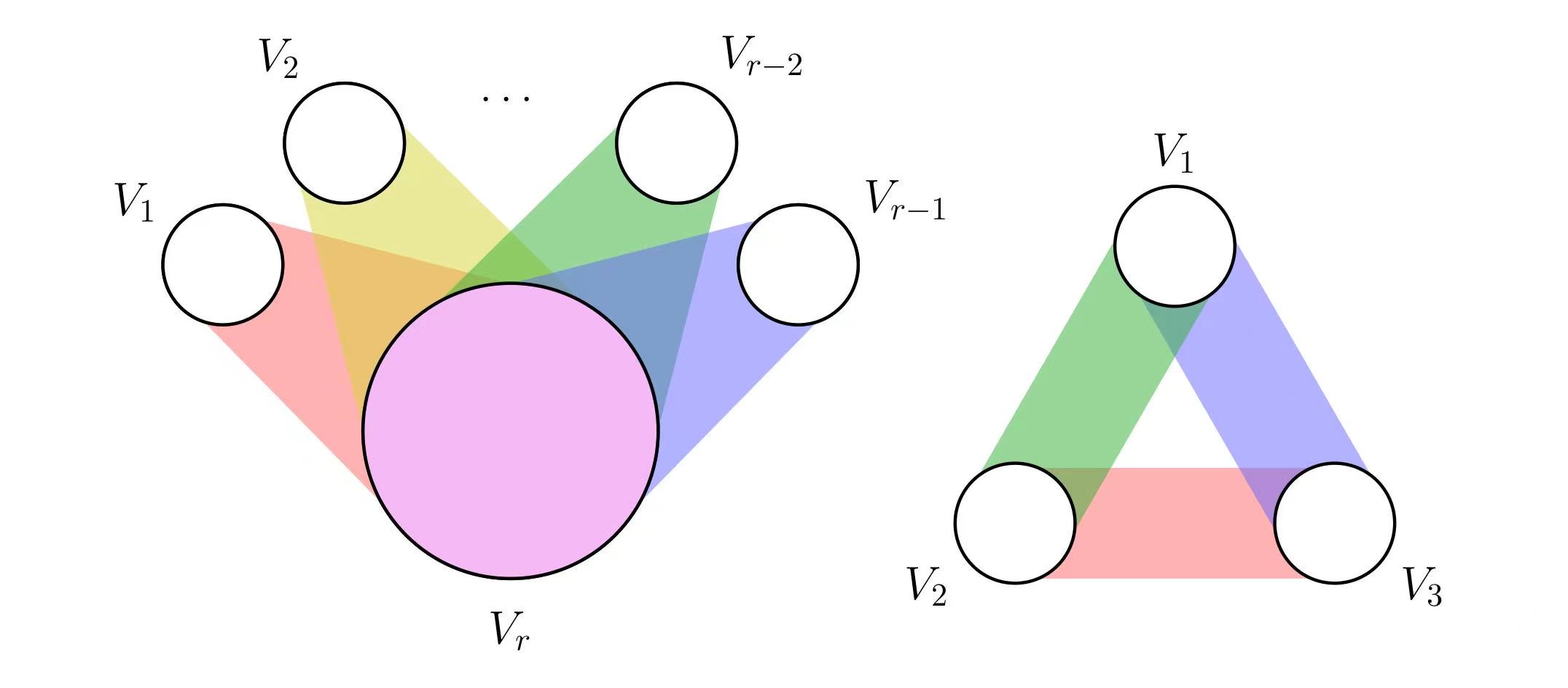}
    \caption{Two extremal constructions for general $r\ge 2$ and for $r=3$}
    \label{fig:extremal_construction_general_r}
\end{figure}


\paragraph{Extremal Construction for \(r=3\):} Assume \( 3 \mid n \), and partition the vertex set into three parts
\[
V(G) = V_1 \cup V_2 \cup V_3,\quad |V_1| = |V_2| = |V_3| = \frac{n}{3}.
\]
Let \( G \) be the complete 3-partite graph with edges only between distinct parts, and color each edge \( uv \) by setting \( \chi(uv) = k \) if and only if \( u, v \notin V_k \), see the right graph in~\cref{fig:extremal_construction_general_r}. Every vertex in this graph has degree \( \frac{2n}{3} \) and is incident to exactly two colors. Now let \( H \subseteq G \) be any Hamilton cycle, and let \( x_i \) denote the number of edges of color \( i \) in \( H \). Each edge is incident to exactly one part, so, for instance, an edge is colored 2 or 3 if and only if it contains exactly one vertex from \( V_1 \). Since \( |V_1| = \frac{n}{3} \), there are exactly \( 2|V_1| = \frac{2n}{3} \) such edges in \( H \), implying \( x_2 + x_3 = \frac{2n}{3} \). Applying the same reasoning to \( V_2 \) and \( V_3 \), we get: \(x_1 + x_2 = \frac{2n}{3},\) and \( x_1 + x_3 = \frac{2n}{3}.\) Solving this system yields \( x_1 = x_2 = x_3 = \frac{n}{3} \), so again \( d_{\chi}(H) = 0 \) for every Hamilton cycle \( H \subseteq G \).

Given the extremal constructions above, it is natural to ask a stability-type question: if a graph \( G \) satisfies \( \delta(G) \leq \frac{r+1}{2r}n \), and all Hamilton cycles in \( G \) exhibit small color-bias, must \( G \), together with its edge-coloring, be structurally close to these extremal examples? First one can readily verify that the underlying graph admits a partition into parts \( V_1, \dots, V_r \) with the following properties:
\begin{itemize}
    \item In the first construction, the induced subgraph \( G[V_r] \) is exactly monochromatic in color \( r \), and hence \((s,r)\)-nearly monochromatic for any \( s \).
    \item In the first construction, each bipartite subgraph \( G[V_i, V_r] \) with \( i \neq r \) is exactly monochromatic in color \( i \), and thus \((s,i)\)-nearly monochromatic. In the second construction for \(r=3\), for each pair \(\{i,j\}\in\binom{[3]}{2}\), the bipartite subgraph \(G[V_{i},V_{j}]\) is monochromatic in color \(k\in [3]\setminus\{i,j\},\) thus \((s,k)\)-nearly monochromatic.
    \item  In the first construction, the subgraph \( G\left[\bigcup_{i \ne r} V_i\right] \) contains no edges and is thus \( s \)-nearly empty. In the second construction for \( r = 3 \), each part \( V_i \) induces an empty subgraph \( G[V_i] \), implying that all such subgraphs are \( s \)-nearly empty.
\end{itemize}
A partial answer to this question was recently provided by Chen, Cheng and Yan~\cite{chen2025colourbiasedhamiltoncyclesrandomly}, who showed that in an \(n\)-vetrex \(r\)-colored graph $G$, if every Hamilton cycle \( H \subseteq G \) satisfies \(d_{\chi}(H) \leq m,\) then there exists a subset \( U \subseteq V(G) \) with \( |U| = \frac{r+1}{2r}n \), and a color \( k \in [r] \), such that the induced subgraph \( G[U] \) is \((2^{9}r^{2}m,k)\)-nearly monochromatic. Moreover, this stability result is a key to their main result in randomly perturbed settings, see~\cite[Theorem~4.1]{chen2025colourbiasedhamiltoncyclesrandomly}, where the minimum degree condition \(\delta(G)\ge \frac{(r+1)n}{2r}\) is optimal. Moreover, recall from~\cref{thm:2021SIDMA} that the minimum degree threshold \( \delta(G) = \frac{r+1}{2r}n \) is tight for ensuring the existence of a Hamilton cycle with prescribed color-bias. Then it is thus natural to expect that the structural stability result of Chen, Cheng and Yan also reflects this threshold. 

Surprisingly, this is not the case: we show that the same type of stability conclusion continues to hold under a strictly weaker minimum degree condition. Such phenomena, where a stability result holds far below the extremal threshold, are of particular interest. One instructive example comes from the classical Tur\'{a}n problem: although the extremal \(K_3\)-free graphs are complete balanced bipartite graphs with minimum degree \(\frac{n}{2}\), a famous stability result of Andr\'{a}sfai, Erd\H{o}s, and S\'{o}s~\cite{1974ErdosSos} states that ensuring a triangle-free graph is bipartite requires only a minimum degree exceeding \(\frac{2n}{5}\), which is strictly below \(\frac{n}{2}\). More precisely, our result shows that any \( n \)-vertex graph with minimum degree slightly above \( \frac{n}{2} \), in which all Hamilton cycles are nearly color-balanced, must essentially look like one of these extremal constructions, up to the relaxation from exact monochromaticity and emptiness to the approximate notions of \((s,k)\)-nearly monochromatic and \( s \)-nearly empty.

\begin{theorem}\label{main-thm}
For any integer \( r \geq 2 \) with \( r \ne 3 \), let \( m, n, s \) be positive integers satisfying \( m < 2^{-6} r^{2} n \) and \( s = 100 r^2 m \). Let \( G \) be an \( n \)-vertex graph with minimum degree \( \delta(G) = \frac{n}{2} + 6r^{2}m \), and let \( \chi : E(G) \to [r] \) be an edge-coloring. Suppose that every Hamilton cycle \( H \subseteq G \) satisfies \( d_{\chi}(H) < m \). Then there exists a color \( k \in [r] \) and a partition \( V(G) = V_1 \cup V_2 \cup \dots \cup V_r \) such that
\[
|V_k| = \frac{r+1}{2r}n \ \text{and} \ |V_i| = \frac{n}{2r} \quad \text{for all } i \ne k,
\]
and the following properties hold:
\begin{itemize}
    \item The subgraph \( G[V_k] \) is \((s,k)\)-nearly monochromatic;
    \item For every \( i \ne k \), the bipartite subgraph \( G[V_i, V_k] \) is \((s,i)\)-nearly monochromatic;
    \item The subgraph \( G\left[ \bigcup_{i \in [r] \setminus \{k\}} V_i \right] \) is \( s \)-nearly empty.
\end{itemize}
\end{theorem}
When \( r = 3 \), the extremal constructions achieving zero color-bias exhibit a qualitatively different behavior from the \( r \ne 3 \) cases. In addition to the typical configuration where one part dominates (that is, one color controls a large fraction of the graph), there exists another extremal example that is fully 3-partite and color-symmetric as we have mentioned above. As a result, when analyzing the structural stability of color-balanced Hamilton cycles for \( r = 3 \), we must allow for both types of extremal behavior. Our stability result thus naturally splits into two cases, each capturing proximity to one of the extremal configurations perfectly.
\begin{theorem}\label{thm:3-colorable}
Let \( m, n, s \) be positive integers satisfying \( m < 2^{-6} \cdot 3^2 \cdot n \) and \( s = 900 m \). Let \( G \) be an \( n \)-vertex graph with minimum degree \( \delta(G) = \frac{n}{2} + 54m \), and let \( \chi : E(G) \to [3] \) be an edge-coloring. Suppose every Hamilton cycle \( H \subseteq G \) satisfies \( d_{\chi}(H) < m \). Then exactly one of the followings occurs:
\begin{enumerate}
\item[\textup{(1)}] There exists a partition \( V(G) = V_1 \cup V_2 \cup V_3 \) with \( |V_1| = |V_2| = |V_3| = \frac{n}{3} \), such that for all distinct \( i, j, k \in [3] \), the following hold:
\begin{itemize}
    \item The subgraph \( G[V_i] \) is \( s \)-nearly empty;
    \item The bipartite subgraph \( G[V_i, V_j] \) is \((s,k)\)-nearly monochromatic.
\end{itemize}

\item[\textup{(2)}] There exists a color \( k \in [3] \) and a partition \( V(G) = V_1 \cup V_2 \cup V_3 \) such that
\[
|V_k| = \frac{2n}{3}, \ \text{and} \ |V_i| = \frac{n}{6} \quad \text{for all } i \in [3]\setminus \{k\},
\]
and the following properties hold:
\begin{itemize}
    \item The subgraph \( G[V_k] \) is \((s,k)\)-nearly monochromatic;
    \item For each \( i \ne k \), the bipartite subgraph \( G[V_i, V_k] \) is \((s,i)\)-nearly monochromatic;
    \item The subgraph \( G\left[ \bigcup_{i \in [3] \setminus \{k\}} V_i \right] \) is \( s \)-nearly empty.
\end{itemize}
\end{enumerate}
\end{theorem}
We remark that the minimum degree condition \( \delta(G) \ge \frac{n}{2} + \Theta(m) \) in our result is essentially optimal. First, by the classical theorem of Dirac~\cite{1952PLMS}, any graph with \( \delta(G) < \frac{n}{2} \) may not contain a Hamilton cycle at all, in which case the notion of color-bias is not even well-defined. Furthermore, the error term \( \Theta(m) \) cannot, in general, be significantly improved: for \( r = 2 \) and large \(m\), there exists a different type of construction of an edge-colored graph \( G \) with minimum degree \( \frac{n}{2} + o(m) \), in which every Hamilton cycle has color-bias strictly less than \( m \), see~\cref{fig:new_construction}. 
\paragraph{A different construction:} 
Let \( r = 2 \), and let \( t = o(m) \). Define a graph \( G \) on \( n \) vertices by partitioning \( V(G) = V_0 \cup V_1 \cup V_2 \), where \( |V_1| = |V_2| = \frac{n - t}{2} \) and \( |V_0| = t \). The edge set is defined as follows:
\begin{itemize}
    \item \( G[V_0 \cup V_2] \) is a complete graph, with all edges colored red;
    \item \( G[V_1] \) is a complete graph, with all edges colored blue;
    \item \( G[V_0, V_1] \) is a complete bipartite graph, with all edges colored blue.
\end{itemize}
Then \( \delta(G) = \frac{n+t}{2} \), and every Hamilton cycle \( H \subseteq G \) has color-bias strictly less than \( m \). Actually each vertex in $V_2$  must be incident to two red edges on $H$, and each vertex in $V_1$ must be incident to two blue edges, therefore $H$ contains at least $\frac{n-t}{2}$ edges of both colors, and $d_{\chi}(H)\leq \frac{t}{2}<m$. This shows that the error term \( \Theta(m) \) in~\cref{main-thm} is tight up to constants when \(r=2\).

\begin{figure}[h!]
    \centering
    \includegraphics[width=0.7\linewidth]{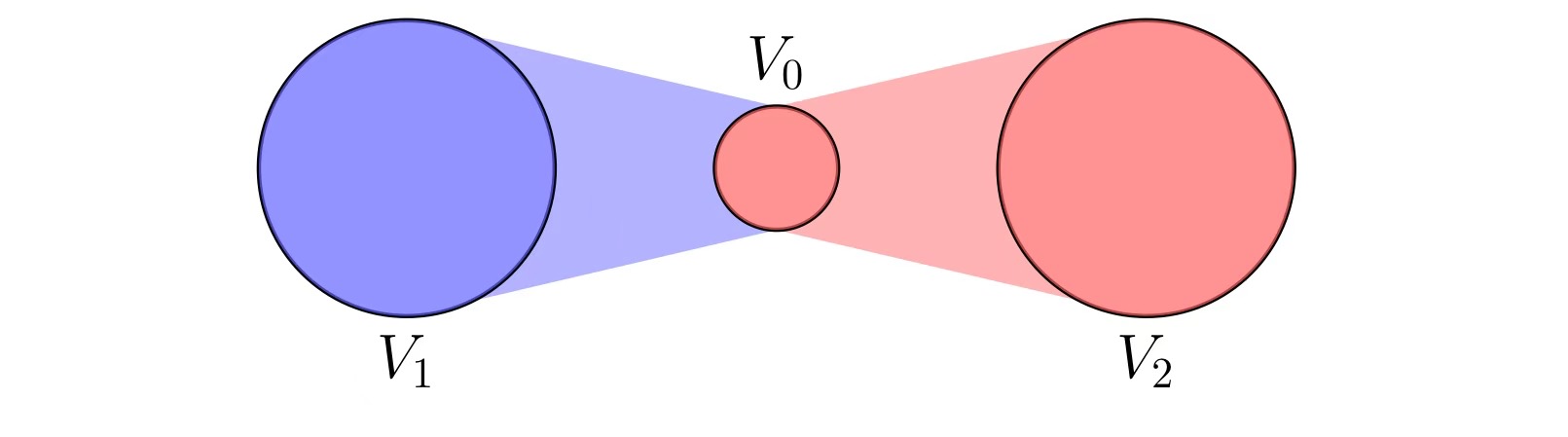}
    \caption{A construction witnessing the optimality of the error term}
\label{fig:new_construction}
\end{figure}

\section{The proofs}
Let \(G\) be a graph on \(n\) vertices with minimum degree \(\delta(G) = \frac{n}{2} + 6r^{2}m\), and let \(\chi\) be an \(r\)-coloring of \(G\) such that for every Hamilton cycle \(H\), we have \(d_{\chi}(H) \leq m\). In a seminal result, Lajos P\'{o}sa~\cite{P} proved that a dense graph can extend any small disjoint path system into a Hamiltonian cycle. More precisely, we will use the following classical result proven in~\cite{P}.
\begin{lemma}[\cite{P}]\label{include-path-forest}
For any integer $C\ge 1$, let $G$ be an graph with minimum degree $\frac{n}{2}+\frac{C+1}{2}$.
Let $L$ be a union of vertex-disjoint paths 
such that $|E(L)|\leq C$, then there exists a Hamilton cycle of $G$ containing $L$ as a subgraph.     
\end{lemma}

\subsection{High-level overview of the proof}
Our goal is to show that any \( n \)-vertex graph \( G \) with minimum degree slightly above \( \frac{n}{2} \), in which every Hamilton cycle is nearly color-balanced, must admit a rigid structural decomposition that prohibits large matchings in the complement of certain color classes. To convey the core insights of our argument, we begin with a brief high-level overview of the proof.

The proof is driven by a small yet powerful configuration, which we call a \emph{bad bowtie}. This five-vertex, six-edge structure is carefully designed to act as a local witness to color imbalance: modifying a Hamilton cycle along a bad bowtie changes the number of edges of a given color by a fixed amount. This sensitivity allows us to rigorously convert local color patterns into global color-bias violations, which is not merely technical. In~\cref{remove-bad-bowties}, we take advantage of~\cref{include-path-forest} to show that the presence of many vertex-disjoint bad bowties would enable us to construct two Hamilton cycles with noticeably different color distributions, contradicting the assumed near-balance. As a consequence, we can eliminate all bad bowties by removing only a negligible number of vertices, resulting in a large induced subgraph \( G' \) that contains no such local imbalances.

The absence of bad bowties in \( G' \) imposes strong local restrictions on how edge colors can be distributed in each vertex's neighborhood. We exploit this rigidity in~\cref{local'} by classifying all vertices into three structural types, labelled \( A(j,k,\ell) \), \( B(k,\ell) \), and \( C(k) \), based on their incident color patterns and neighborhood structure. This classification reveals the internal constraints each vertex must satisfy, and lays the foundation for a full structural description of \( G' \) in~\cref{main-structure}.

The final step is to understand how these vertex types interact on a global scale. When \( r \neq 3 \), we show that all vertices in \( G' \) must belong to a configuration dominated by a single color \( k \), leading to the matching-exclusion structure asserted in~\cref{main-thm}. When \( r = 3 \), a second, highly symmetric configuration becomes possible: the vertex set decomposes into three parts, each interacting in a tightly controlled way with the others via fixed color patterns, as captured in~\cref{thm:3-colorable}. In both cases, the resulting structural decomposition ensures that any large matching avoiding a given color must intersect the small set of deleted vertices. As a result, every such forbidden matching is bounded in size.

\subsection{The bad bowties and removal }
We begin by introducing a key structural configuration, referred to as a \emph{bowtie}, which plays a central role in our proof. This configuration has been extensively studied in the literature~\cite{1995Erdos,2020EUJC,LFP2025EUJC}.

\begin{defn}
A \emph{bowtie} \(B\) is a graph with
\[
V(B) = \{v_1, v_2, v_3, v_4, v_5\}, \quad
E(B) = \{v_1v_2, v_1v_3, v_1v_4, v_1v_5, v_2v_3, v_4v_5\}.
\]
We also denote the bowtie by the ordered sequence \(v_1v_2v_3v_4v_5\).
\end{defn}

Given a color $k\in [r]$ and a bowtie $B:=v_1v_2v_3v_4v_5$, we 
define the \emph{color-counting function} as 
$$f(B,k)=f(v_1v_2v_3v_4v_5,k)=\mathbf{1}_k(v_1v_2)+\mathbf{1}_k(v_1v_3)+\mathbf{1}_k(v_4v_5)-\mathbf{1}_k(v_1v_4)-\mathbf{1}_k(v_1v_5)-\mathbf{1}_k(v_2v_3),$$
where \(\mathbf{1}_k : E(G) \to \{0, 1\}\) is the indicator function defined by
\[
\mathbf{1}_k(e) =
\begin{cases}
1 & \text{if } \chi(e) = k, \\
0 & \text{otherwise}.
\end{cases}
\]
We say that a bowtie \(B\) is \emph{\(k\)-bad} if \(f(B, k) \ne 0\), and \emph{bad} if it is \(k\)-bad for some color \(k \in [r]\). In~\cref{fig:bowtie_gallery}, we display all non-isomorphic bowtie configurations that are not bad, as well as several representative examples of bad bowties.

\begin{figure}[h!]
    \centering
    \includegraphics[width=1\linewidth]{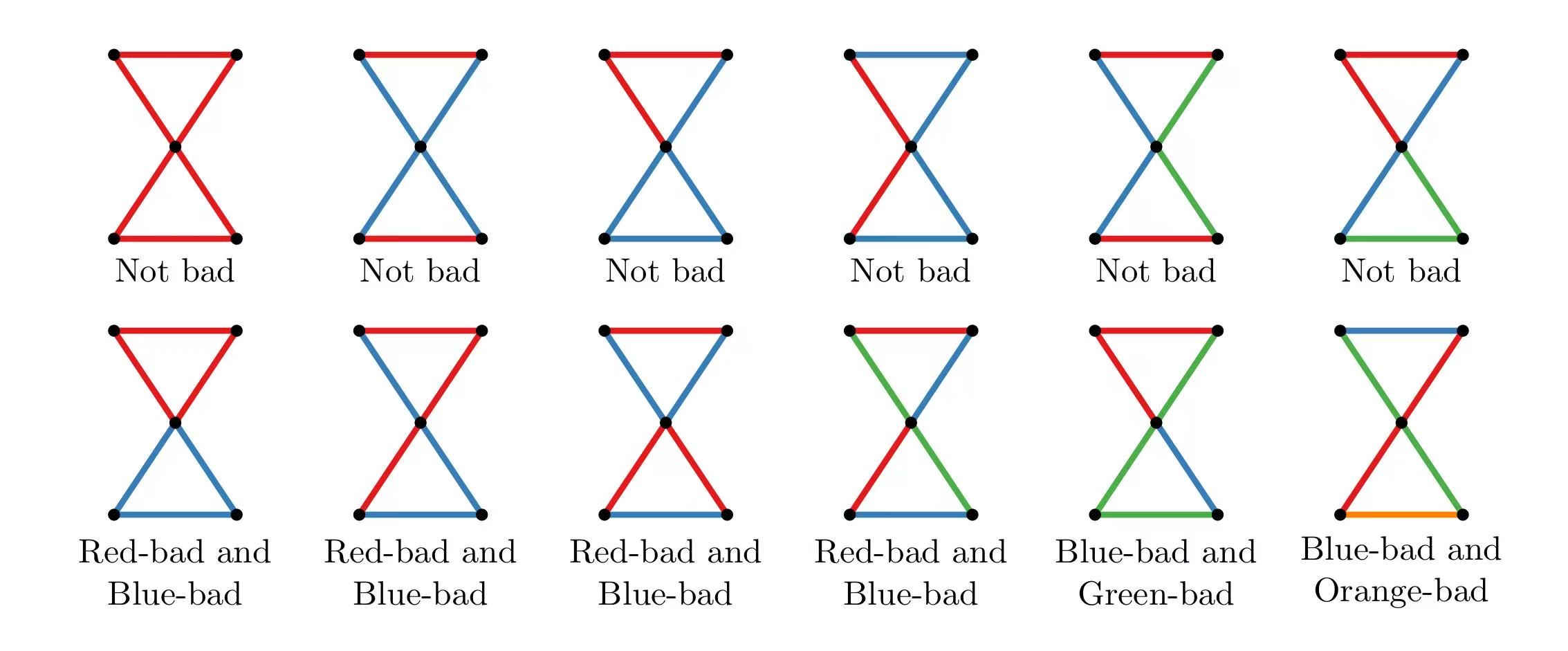}
    \caption{The top row illustrates all structurally distinct bowties that are not bad, while the bottom row displays representative examples of bad bowties}
\label{fig:bowtie_gallery}
\end{figure}

One key structural observation in our proof is that we can always find a large induced subgraph of \(G\) that contains no bad bowties.

\begin{prop}\label{remove-bad-bowties}
There exists an induced subgraph \(G' \subseteq G\) such that \(|V(G) \setminus V(G')| \leq 10rm\), and \(G'\) contains no bad bowties.
\end{prop}

\begin{proof}[Proof of~\cref{remove-bad-bowties}]
Let \(\mathcal{B}\) be a maximal collection of vertex-disjoint bad bowties in \(G\), and let \(|\mathcal{B}| = t\). We will show that \(t \leq 2rm\).

Suppose for contradiction that \(t > 2rm\). By the pigeonhole principle, there exists a color \(k \in [r]\) such that at least \(2m\) bowties in \(\mathcal{B}\) are \(k\)-bad. Take \(2m\) of these bowties and denote them by \(B_i = a_i b_i c_i d_i e_i\) for \(1 \leq i \leq 2m\). By definition, we may assume that \(f(B_i, k) > 0\) for each \(i\), possibly after relabeling the vertices of each \(B_i\) as \(a_i d_i e_i b_i c_i\) if needed.

Now define two subgraphs \(L_1\) and \(L_2\) of \(G\) with vertex set
\[
V(L_1) = V(L_2) = \bigcup_{i=1}^{2m} \{a_i, b_i, c_i, d_i, e_i\},
\]
and edge sets
\[
E(L_1) = \{a_i b_i, a_i c_i, d_i e_i : 1 \leq i \leq 2m\}, \quad
E(L_2) = \{b_i c_i, a_i d_i, a_i e_i : 1 \leq i \leq 2m\}.
\]
Each of \(L_1\) and \(L_2\) is a union of vertex-disjoint paths. Since \(|E(L_1)| = 6m \leq 12r^{2}m-1\) when $r\ge 2$, we then apply~\cref{include-path-forest} to obtain a Hamilton cycle \(H_1\) in \(G\) such that \(L_1 \subseteq H_1\). Now define \(H_2 := (H_1 \cup E(L_2)) \setminus E(L_1)\). Notice that this replacement preserves the Hamiltonicity of \(H_1\), and \(H_2\) is also a Hamilton cycle.

By the definition of the function \(f(\cdot, k)\), we have:
\[
|\chi^{-1}(k) \cap E(H_1)| - |\chi^{-1}(k) \cap E(H_2)|
= \sum_{i=1}^{2m} f(B_i, k) \geq 2m.
\]
It follows that either
\[
|\chi^{-1}(k) \cap E(H_1)| \geq \frac{n}{r} + m \quad \text{or} \quad
|\chi^{-1}(k) \cap E(H_2)| \leq \frac{n}{r} - m,
\]
contradicting the assumption that \(d_{\chi}(H)<m\) every Hamilton cycle $H$. Therefore, \(t \leq 2rm\).

Finally, define the induced subgraph \(G' \subseteq G\) by
\[
V(G') = V(G) \setminus \left(\bigcup_{B \in \mathcal{B}} V(B)\right).
\]
Since each bowtie has five vertices, we have
\[
|V(G) \setminus V(G')| \leq 5t \leq 10rm.
\]
Moreover, by the maximality of \(\mathcal{B}\), the graph \(G'\) contains no bad bowties.

\end{proof}

\subsection{Structure of subgraphs without bad bowties}
By~\cref{remove-bad-bowties}, we may take an induced subgraph \( G' \subseteq G \) with \( |V(G) \setminus V(G')| \le 10rm \), such that \( G' \) contains no bad bowties. For each vertex \( v \in V(G') \), define
\[
L(v) := \{ k \in [r] : \textup{there\ exists\ } u \in V(G') \textup{ such that } uv \in E(G') \text{ and } \chi(uv) = k \}.
\]
This set records the set of colors that appear on edges incident to \( v \) in \( G' \). For each \( k \in L(v) \), define
\[
N^k(v) := \left\{ u \in V(G') : uv \in E(G') \text{ and } \chi(uv) = k \right\},
\]
\[
N^{\neq k}(v) := \left\{ u \in V(G') : uv \in E(G') \text{ and } \chi(uv) \ne k \right\}.
\]
Moreover, for distinct \( k, \ell \in L(v) \), let \( N^{k,\ell}(v) \) denote the bipartite subgraph of \( G' \) induced by the bipartition \( (N^k(v), N^\ell(v)) \).

Given a vertex \(v \in V(G')\), recall that \(L(v)\) denotes the set of colors appearing on edges incident to \(v\) in \(G'\), and \(N^k(v)\) denotes the set of neighbors of \(v\) via edges colored \(k\). We define the following vertex classes in \(G'\) based on their colored neighborhood structure:

\medskip
\noindent
\textbf{Type A} (\(A(j,k,\ell)\)): A vertex \(v\in V(G')\) belongs to \(A(j,k,\ell)\) if
\[
L(v) = \{j,k\}, \ N^j(v) \text{ and } N^k(v) \text{ are independent sets},\ \text{and } \chi(e) = \ell \text{ for all } e \in E(G'[N^j(v) \cup N^k(v)]).
\]

\medskip
\noindent
\textbf{Type B} (\(B(k,\ell)\)): A vertex \(v\in V(G')\) belongs to \(B(k,\ell)\) if
\[
L(v) = \{k\},\ \text{and } \chi(e) = \ell \text{ for all } e \in E(G'[N^k(v)]).
\]

\medskip
\noindent
\textbf{Type C} (\(C(k)\)): A vertex \(v\in V(G')\) belongs to \(C(k)\) if
\begin{itemize}
\item all edges in $E[N^k(v)]$ have color \(k\), that is, \(\chi(u_1u_2) = k\) for all \(u_1,u_2 \in N^k(v)\),
\item the set \(N^{\neq k}(v)\) is independent in \(G'\),
\item and for every \(\ell \in L(v) \setminus \{k\}\), the bipartite subgraph \(N^{k,\ell}(v)\) induced by \((N^k(v), N^\ell(v))\) is monochromatic in color \(\ell\).
\end{itemize}

We now show that every vertex in \( G' \) falls into exactly one of the structural types defined above.

\begin{prop}\label{local'}
For every vertex \( v \in V(G'), \) exactly one of the followings holds:
\begin{enumerate}
    \item[\textup{(1)}] \( v \in A(j,k,\ell) \) for some distinct \( j, k, \ell \in [r] \);
    \item[\textup{(2)}] \( v \in B(k,\ell) \) for some \( k, \ell \in [r] \);
    \item[\textup{(3)}] \( v \in C(k) \) for some \( k \in [r] \).
\end{enumerate}
\end{prop}

\begin{proof}[Proof of~\cref{local'}]
Fix a vertex \( v \in V(G') \), and consider the induced subgraph \( G'[N(v)] \). For each edge \( uw \in E(G'[N(v)]) \), we classify \( uw \) according to the coloring of the triangle \( vuw \). More precisely, for distinct colors \( j, k, \ell \in [r] \), define:
\begin{itemize}
    \item \( uw \) is of type \( a(j,k,\ell) \) if and only if \( \{\chi(vu), \chi(vw)\} = \{j,k\} \) and \( \chi(uw) = \ell \);
    \item \( uw \) is of type \( b(k,\ell) \) if and only if \( \chi(vu) = \chi(vw) = k \) and \( \chi(uw) = \ell \);
    \item \( uw \) is of type \( c(k) \) if and only if either \( \chi(vu) = \chi(vw) = \chi(uw) = k \), or \( \{\chi(vu), \chi(vw)\} = \{k,\ell\} \) and \( \chi(uw) = \ell \) for some \( \ell \ne k \).
\end{itemize}
Notice that every edge $uw$ must belong to one of the above types, therefore it is well-defined. Furthermore, we claim that all edges in \( E(G'[N(v)]) \) must have the same type. To show this, consider any two edges \( u_1u_2, w_1w_2 \in E(G'[N(v)]) \). Since we have already listed all of the possible bowties which are not bad in~\cref{fig:bowtie_gallery}, then one can easily check that if \( u_1u_2 \) and \( w_1w_2 \) are vertex-disjoint, the two edges \(u_{1}u_{2}\) and \(w_{1}w_{2}\) must have the same type, otherwise they together with \(v\) will form a bad bowtie.

Now consider the case where \( u_1u_2 \) and \( w_1w_2 \) are not necessarily disjoint. By the minimum degree condition, each vertex \( u \in N(v) \cap V(G') \) satisfies
\[
|N(u) \cap N(v) \cap V(G')| \ge 12r^{2}m-10rm,
\]
we can greedily find a matching \( M \subseteq G'[N(v)] \) of size at least \( 6r^{2}m-5rm > 5 \) since $r\ge 2$. In particular, there exists an edge \( e \in E(M) \) that is disjoint from both \( u_1u_2 \) and \( w_1w_2 \). Then by the above claim, this edge \( e \) must be of the same type as both \( u_1u_2 \) and \( w_1w_2 \), implying that \( u_1u_2 \) and \( w_1w_2 \) are also of the same type.

Thus, all edges in \( G'[N(v)] \) are of the same type. A direct inspection of the definitions verifies that \( v \in A(j,k,\ell) \) (resp. \( B(k,\ell) \), \( C(k) \)) if and only if all edges in \( G'[N(v)] \) are of type \( a(j,k,\ell) \) (resp. \( b(k,\ell) \), \( c(k) \)). This completes the proof.
\end{proof}

Having established that every vertex in \( G' \) belongs to one of the types \( A(j,k,\ell) \), \( B(k,\ell) \), or \( C(k) \), we now proceed to classify the global structure of \( G' \) based on how these types can coexist. Remarkably, the local color constraints encoded in these vertex types lead to a complete structural dichotomy, as captured in the following theorem.

\begin{prop}\label{main-structure}
The following structural dichotomy holds for \(G'\):
\begin{enumerate}
\item[\textup{(1)}] If \(r \ge 2\) and \(r \ne 3\), then there exists some \(k \in [r]\) such that
\[
V(G') = \left( \bigcup_{\ell \in [r] \setminus \{k\}} B(k,\ell) \right) \cup C(k).
\]

\item[\textup{(2)}] If \(r = 3\), then either
\[
V(G') = A(1,2,3) \cup A(2,3,1) \cup A(3,1,2),
\]
or there exists some \(k \in [3]\) such that
\[
V(G') = \left( \bigcup_{\ell \in [3] \setminus \{k\}} B(k,\ell) \right) \cup C(k).
\]
\end{enumerate}
\end{prop}

\begin{proof}[Proof of~\cref{main-structure}]
For any \( uv \in E(G') \), since \( |N(v) \cap N(u) \cap V(G')| \geq 12r^{2}m-10rm>0 \),
we can find a vertex \( w \in N(v) \cap N(u) \cap V(G') \).
We will show that, given the type of $v$, 
we can determine all possibilities of the type of $u\in N(v)\cap V(G')$, as follows:
\begin{itemize}
\item If $v\in A(j,k,\ell)$, then $\{\chi(uv),\chi(vw)\}=\{j,k\}$, $\chi(uw)=\ell$,
and $u\in A(k,\ell,j)\cup A(\ell,j,k)$.
\item If $v\in B(k,\ell)$, then $\chi(vu)=\chi(vw)=k$, $\chi(uw)=\ell$, and $u\in C(\ell)$.
\item If \( v \in C(k) \), then either \( \chi(uv) = \chi(vw) = \chi(uw) = k \), in which case \( u \in C(k) \), or \( \{\chi(uv), \chi(vw)\} = \{k, \ell\} \) and \( \chi(uw) = \ell \) for some \( \ell \ne k \), in which case \( u \in C(k) \) or \( u \in B(k,\ell) \).
\end{itemize}
We now divide the proof into two cases:

\begin{Case}
\item Suppose \( A(j,k,\ell) = \emptyset \) for all distinct \( j,k,\ell \in [r] \).
By the analysis above, we must have \( C(k) \ne \emptyset \) for some \( k \in [r] \).
Pick any \( v \in C(k) \).
Then \( N(v) \cap G' \subseteq \left( \bigcup_{\ell \in [r] \setminus \{k\}} B(k,\ell) \right) \cup C(k) \),
and for any \( \ell \ne k \) and \( u \in B(k,\ell) \),
we have \( N(u) \cap G' \subseteq C(k) \).
It follows that the connected component of \( G' \) containing \( v \) is entirely contained in \( \left( \bigcup_{\ell \in [r] \setminus \{k\}} B(k,\ell) \right) \cup C(k) \).
By~\cref{remove-bad-bowties}, \(\delta(G')\ge \delta(G)-10rm>\frac{|V(G')|}{2}\), then \(G'\) is also connected, we can conclude that
\[
V(G') = \left( \bigcup_{\ell \in [r] \setminus \{k\}} B(k,\ell) \right) \cup C(k).
\]
\item Suppose \( A(j,k,\ell) \ne \emptyset \) for some distinct \( j,k,\ell \in [r] \).
For any \( v \in A(j,k,\ell) \), we have \( N(v) \cap G' \subseteq A(k,\ell,j) \cup A(\ell,j,k) \).
Thus, the connected component of \( G' \) containing \( v \) is contained in
\[
A(j,k,\ell) \cup A(k,\ell,j) \cup A(\ell,j,k).
\]
Again by~\cref{remove-bad-bowties}, \(\delta(G')\ge \delta(G)-10rm>\frac{|V(G')|}{2}\), we can see this connected component contains all vertices of \( V(G') \). Observe that each of these types is defined solely in terms of the colors \( j,k,\ell \), and any vertex of type \( A(j,k,\ell) \) is adjacent only to vertices of the other two types in this triple.
As a result, no vertex in \( G' \) can be adjacent to an edge colored with any color outside of \( \{j,k,\ell\} \), which implies that \( r = 3 \). Hence,
\[
V(G') = A(1,2,3) \cup A(2,3,1) \cup A(3,1,2).
\]
\end{Case}
This finishes the proof.
\end{proof}
\subsection{Finishing the proofs of~\cref{main-thm} and~\cref{thm:3-colorable}}
Based on the discussion in the above part, we first consider the case when \(r = 3\) and \(G'\) consists entirely of type \(A\) vertices; that is,
\[
V(G') = A(1,2,3) \cup A(2,3,1) \cup A(3,1,2),
\]
as in the statement~(2) of~\cref{main-structure}. By~\cref{remove-bad-bowties}, we can write
\[
V(G) = V_0 \cup A(1,2,3) \cup A(2,3,1) \cup A(3,1,2),
\]
where \(|V_0| \le 10rm\) consists of vertices involved in bad bowties.

Let \(H\) be any Hamilton cycle in \(G\). Since \(H\) contains at most two edges incident to each vertex in \(V_0\), at most \(2|V_0|\) vertices in \(A(1,2,3)\) may be adjacent in \(H\) to a vertex in \(V_0\), which yields at least \(|A(1,2,3)| - 2|V_0|\) vertices in \(A(1,2,3)\) are adjacent to vertices in \(A(2,3,1) \cup A(3,1,2)\). Moreover, since \(A(1,2,3)\) forms an independent set in $G'$, we can see the number of such cross-part edges in \(H\) is at least \(2(|A(1,2,3)| - 2|V_0|).\) On the other hand, the total number of vertices outside \(A(1,2,3) \cup V_0\) is
\[
|A(2,3,1)| + |A(3,1,2)| = n - |A(1,2,3)| - |V_0|.
\]
Combining both, we get:
\[
2(|A(1,2,3)| - 2|V_0|) \le n - |A(1,2,3)| - |V_0|,
\]
which simplifies to:
\[
|A(1,2,3)| \le \frac{n}{3} + |V_0| \le \frac{n}{3} + 10rm.
\]
By symmetry, the same bound holds for \(A(2,3,1)\) and \(A(3,1,2)\). Since the three sets partition \(V(G')\), we conclude that
\[
|A(i,j,k)| \ge \frac{n}{3} - 30rm \ \text{for each distinct triple } (i,j,k) \in [3].
\]

We now define three equal-sized vertex classes \(V_1, V_2, V_3\) such that
\[
|V_1| = |V_2| = |V_3| = \frac{n}{3}, \ \text{and} \ |V_i \, \triangle \, A(j,k,i)| \le 30rm
\]
for each distinct \(j,k,i \in [3]\). Let us define the error sets \(W_i := V_i \setminus A(j,k,i)\).

Since each vertex \(v \in A(2,3,1)\) satisfies that \(N^2(v)\) and \(N^3(v)\) are independent sets, and since \(A(2,3,1)\) corresponds closely to \(V_1\), it follows that any matching contained in \(E(V_1)\) must lie entirely within the symmetric difference \(W_1 = V_1 \setminus A(2,3,1)\). Thus, for any matching \(M \subseteq E(V_1)\),
\[
|V(M)| \le |W_1| \le 30rm < 100 r^2 m.
\]

Similarly, all edges in \(E(A(1,2,3), A(3,1,2))\) are colored with color \(1\), so any matching in \(E(V_2, V_3)\) that avoids color \(1\) must lie in the error set \(W_2 \cup W_3\). Hence, for any matching \(M' \subseteq E(V_2, V_3) \setminus \chi^{-1}(1)\),
\[
|V(M')| \le |W_2| + |W_3| \le 60rm < 100 r^2 m.
\]
The same argument applies to all other index triples \((i,j,k)\), completing the verification of condition~(1) in~\cref{thm:3-colorable}.

It remains to prove~\cref{main-thm} (for \(r \ne 3\)) and the second part of~\cref{thm:3-colorable}. By~\cref{remove-bad-bowties} and~\cref{main-structure}, we can partition the vertex set as
\[
V(G) = V_0 \cup C(k) \cup \left(\bigcup_{\ell \in [r] \setminus \{k\}} B(k,\ell)\right),
\]
for some \(k \in [r]\), where \(|V_0| \le 10rm\) consists of vertices involved in bad bowties.

Let \(H\) be any Hamilton cycle in \(G\), and fix \(\ell \ne k\). Each vertex in \(B(k,\ell)\) has all incident edges (within \(G'\)) colored \(\ell\), and contributes two \(\ell\)-colored edges to \(H\), unless it is adjacent to a vertex in \(V_0\), in which case it may contribute fewer. Also notice that \(H\) contains at most two edges incident to each vertex in \(V_{0}\), we can see the total number of \(\ell\)-colored edges in \(H\) satisfies
\[
|\chi^{-1}(\ell) \cap E(H)| \ge 2(|B(k,\ell)| - 2|V_0|),
\]
which implies
\[
|B(k,\ell)| \le \tfrac{1}{2} |\chi^{-1}(\ell) \cap E(H)| + 2|V_0| \le \tfrac{n}{2r} + \tfrac{m}{2} + 20rm \le \tfrac{n}{2r} + 30rm.
\]

Conversely, each \(\ell\)-colored edge in \(H\) must be incident to either a vertex in \(B(k,\ell)\) or in \(V_0\), so we also have
\[
|\chi^{-1}(\ell) \cap E(H)| \le 2(|B(k,\ell)| + |V_0|),
\]
which yields
\[
|B(k,\ell)| \ge \tfrac{1}{2} |\chi^{-1}(\ell) \cap E(H)| - |V_0| \ge \tfrac{n}{2r} - \tfrac{m}{2} - 10rm \ge \tfrac{n}{2r} - 20rm.
\]

Now compute the size of \(C(k)\) by subtracting the other parts:
\[
|C(k)| = n - \sum_{\ell \ne k} |B(k,\ell)| - |V_0| \ge n - (r-1)\left( \tfrac{n}{2r} + 30rm \right) - 10rm = \tfrac{r+1}{2r}n - 40r^2m,
\]
and similarly,
\[
|C(k)| \le n - (r-1)\left( \tfrac{n}{2r} - 20rm \right) = \tfrac{r+1}{2r}n + 20r^2m.
\]

Hence, we can define sets \(V_i\) for \(i \in [r]\) such that
\[
|V_k| = \tfrac{r+1}{2r}n, \ |V_\ell| = \tfrac{1}{2r}n \text{ for all } \ell \ne k,
\]
and
\[
|V_k \triangle C(k)| \le 40r^2m, \quad |V_\ell \triangle B(k,\ell)| \le 40rm \text{ for all } \ell \ne k.
\]

Let \(W_k := V_k \setminus C(k)\), and for each \(\ell \ne k\), define \(W_\ell := V_\ell \setminus B(k,\ell)\). We now verify the matching exclusion conditions, which yield the structural properties stated in the theorems.

First, since all edges within \(C(k)\) have color \(k\) by definition of type \(C(k)\), any matching \(M_{1} \subseteq E(V_k) \setminus \chi^{-1}(k)\) must be entirely contained in \(W_k\), and thus
\[
|V(M_{1})| \le |W_k| \le 40r^2m < 100 r^2 m.
\]
Similarly, for any \(\ell \ne k\), the bipartite graph \(E(V_\ell, V_k)\) contains only \(\ell\)-colored edges from \(B(k,\ell)\) to \(C(k)\). By definition of \(B(k,\ell)\), removing color \(\ell\) eliminates all these edges. Hence, any matching \(M_{2}\) in \(E(V_\ell, V_k) \setminus \chi^{-1}(\ell)\) is contained in \(W_\ell \cup W_k\), giving
\[
|V(M_{2})| \le |W_\ell| + |W_k| \le 40r^2m + 40rm < 100 r^2 m.
\]
Finally, all edges in \(E(V(G) \setminus V_k)\) are contained in the disjoint union of \(B(k,\ell)\) and \(W_\ell\), and removing each \(\ell\)-color eliminates the edges in \(B(k,\ell)\). Thus, any matching \(M_{3}\) in \(E(V(G) \setminus V_k)\) must lie in the union of the \(W_\ell\), and we conclude again:
\[
|V(M_{3})| \le \sum_{\ell \ne k} |W_\ell| \le 40r^2m<100 r^2 m.
\]
Therefore, all matching bounds required in~\cref{main-thm} (and the second part of~\cref{thm:3-colorable}) hold, completing the proof.

\section{Concluding remarks}
In this paper, we investigated Hamilton cycles in edge-colored graphs with nearly balanced color patterns and established a sharp structural stability result. Specifically, we proved that if every Hamilton cycle in an $n$-vertex graph with minimum degree exceeding $ \frac{n}{2} + \Theta(m) $ has color-bias at most $m$, then the graph must closely resemble one of the known extremal constructions. A key ingredient in our proof is the notion of a \emph{bad bowtie}, a small local configuration that certifies color imbalance. By analyzing the interaction between such local structures and global Hamiltonicity, we derived strong constraints on the overall graph structure.

Our work raises several natural questions for future research. While the degree condition $ \delta(G) \ge \frac{n}{2} + \Theta(m) $ is optimal for ensuring the existence of Hamilton cycles, it remains unclear whether the additive term $\Theta(m)$ is also necessary for structural stability when $r \ge 3$. For $r = 2$, we answer this affirmatively by constructing a graph with minimum degree $ \frac{n}{2} + o(m) $ in which every Hamilton cycle has color-bias less than $m$, yet the structure deviates from the known extremal forms. It is also natural to pursue stronger forms of stability. In particular, when $r = 2$, does every graph with $ \delta(G) \ge \frac{n}{2} + o(m) $ and small color-bias necessarily resemble either the classical extremal construction of Freschi, Hyde, Lada, and Treglown, or our new example?

Moreover, we suspect that the structural stability of color-balanced Hamilton cycles is not solely governed by the minimum degree condition. Motivated by this, we propose the following conjecture, which replaces the minimum degree assumption with a more robust, local resilience condition.
\begin{conj}
There exists a sufficiently large constant $C > 0$ such that the following holds. Let $G$ be an $n$-vertex graph with an edge-coloring using $r$ colors. Suppose that for every subset $U \subseteq V(G)$ with $|U| \le Cm$, the induced subgraph $G\setminus U$ contains a Hamilton cycle, and that every Hamilton cycle in $G$ has color-bias less than $m$. Then $G$ must structurally resemble one of the known extremal constructions.
\end{conj}

\section*{Acknowledgement}
The authors would like to thank Prof. Jie Han and Prof. Lior Gishboliner for their valuable comments. This work was completed during Wenchong Chen's visit to the ECOPRO group at IBS. He gratefully acknowledges the organizers of the ECOPRO 2025 Student Research Program. Wenchong Chen also thanks Xinbu Cheng and Zhifei Yan for insightful discussions.

\bibliographystyle{abbrv}
\bibliography{Hami}

\end{document}